\documentclass[10pt, preprint]{elsarticle}

\usepackage{amssymb,latexsym,amsmath,epsfig} 

\usepackage{hyperref}
\usepackage[amsthm,thmmarks]{ntheorem}
\usepackage{enumerate}
\usepackage[all,dvips,arc,arrow, matrix, curve]{xy}
\usepackage[poorman]{cleveref}

\newtheorem{theorem}{Theorem}[section]
\newtheorem{lemma}[theorem]{Lemma}
\newtheorem{proposition}[theorem]{Proposition}

\newtheorem{definition}[theorem]{Definition}
\newtheorem{remark}[theorem]{Remark}

\theoremstyle{definition}
\newtheorem{example}[theorem]{Example}
\theoremstyle{definition}
\newtheorem{algo}[theorem]{Algorithm}
\Crefname{algo}{Algorithm}{Algorithms}

\newcommand{\N}{\mathbb{N}}
\newcommand{\Z}{\mathbb{Z}}

\newcommand{\R}{\mathbb{R}}

\newcommand{\calL}{\mathcal{L}}

\newcommand{\calD}{\mathcal{D}}
\newcommand{\calC}{\mathcal{C}}
\newcommand{\V}{\mathcal{V}}
\newcommand{\w}{\omega}
\newcommand{\Rhyp}{\R^{n\times n}_{hyp}}
\newcommand{\Zhyp}{\Z^{n\times n}_{hyp}}

\newcommand{\Aut}{\operatorname{Aut}}

\newcommand{\Stab}{\operatorname{Stab}}
\newcommand{\eps}{\varepsilon}
\newcommand{\GL}{\operatorname{GL}}

\newcommand{\Vgt}{\V_1^{>0}}
\newcommand{\Wgt}{\V_2^{>0}}
\newcommand{\Vgeq}{\V_1^{\geq 0}}
\newcommand{\Wgeq}{\V_2^{\geq 0}}
\newcommand{\Vigt}{\V_i^{>0}}
\newcommand{\Vigeq}{\V_i^{\geq 0}}

\newcommand{\FF}{\mathbb{F}}

\newcommand{\Rsym}{\mathbb{R}^{n\times n}_{sym}}
\newcommand{\Rsgt}{\mathbb{R}^{n\times n}_{>0}}

\newcommand{\Watson}{\operatorname{Watson}}
\newcommand{\FH}{\mathfrak{H}}

\title{Automorphism Groups of Hyperbolic Lattices}
\author{Michael H. Mertens\footnote{The results in this paper are partially contained in the author's master's thesis \cite{Mert12} written under the supervision of Prof. Dr. Gabriele Nebe at Lehrstuhl D f\"ur Matematik, RWTH Aachen University, Templergraben 64, D-52062 Aachen, Germany}{}\footnote{The author's research is supported by the DFG Graduiertenkolleg 1269 "Global Structures in Geometry and Analysis"}\\
{\it Mathematisches Institut der Universit\"at zu K\"oln,}\\
{\it Weyertal 86-90,} \\
{\it D-50931 K\"oln, Germany}\\
{\it Phone: $+49-(0)221/4704333$}\\
{\textit{E-Mail:} \tt mmertens@math.uni-koeln.de}\\
\url{http://www.mi.uni-koeln.de/~mmertens}
}
\begin{document}
\maketitle

\centerline{\bf Abstract}

\noindent
Based on the concept of dual cones introduced by J.~Opgenorth we give an algorithm to compute a generating system of the group of automorphisms of an integral lattice endowed with a hyperbolic bilinear form.


\section{Introduction}
In his famous paper  \emph{Nouvelles applications des param\`{e}tres continus \`{a} la th\'{e}orie des formes quadratiques} \cite{Vor}, G.~F. Voronoi presented an algorithm to enumerate (up to scaling) all perfect quadratic forms in a given dimension $n$. The general idea for that was to compute a face-to-face tesselation of a certain cone in the space of symmetric endomorphisms of $\R^n$ based on the pyramides induced by the shortest vectors of a perfect quadratic form. 

Generalizing Voronoi's ideas, M. Koecher came up with the concept of \emph{self-dual cones} or, as he called them, \emph{positivity domains} \cite{Koe57,Koe60} to obtain an alternative to the reduction theory of quadratic forms due to H. Minkowski. 

Another slight generalization of these ideas to so called \emph{dual cones} was then suiting for J. Opgenorth to find an algorithm to determine a generating system for the normalizer $N_{\GL_n(\Z)}(G)$ of a finite unimodular group $G$ of degree $n$ (cf. \cite{Opg01}), which is an essential tool e.g. dealing with crystallographic space groups (cf. \cite{OPS}).

The goal here is to use these methods to derive an algorithm which gives a generating system of the automorphism group of an integral hyperbolic lattice\footnote{In the literature often referred to as a Lorentzian lattice} $(L,\Phi)$ (see Notation below). 

There exists an algorithm due to E. B. Vinberg \cite{Vin72} to construct the maximal normal subgroup of the automorphism group of a hyperbolic lattice that is generated by reflections. But this algorithm terminates if and only if this reflection group has finite index in the full automorphism group. Lattices with this property, so-called \emph{reflective} lattices, are very rare: For example if we consider the lattice $\Z^n$ together with the bilinear forms induced by matrices 
\[H^{(d)}_n:=diag(-d,1,\dots ,1)\in\Z^{n\times n}\text{ for }d>0,\]
then it is known that $(\Z^n,H^{(1)}_n)$ is reflective if and only if $n\leq 19$ (cf. \cite{Vin72}). According to the classification of reflective hyperbolic lattices of rank $3$ by D. Allock in \cite{Lorentz}, the highest prime divisor of the discriminant of such a lattice is $97$ (cf. \cite[p. 24]{Lorentz}. Thus reflective lattices can neither occur in high dimensions nor for high discriminants.

To the author's knowledge so far there is no algorithm known to determine the automorphism group of a general hyperbolic lattice. The algorithm presented in this paper does at least not have theoretical limitations although in practice it can only handle lattices of small ranks and moderate discriminants (see \Cref{secPerformance}).

The paper will be organized as follows: In \Cref{subprelim} we recall the basic definitions and key results about dual cones from \cite{Opg01} which give a general method to determine generating systems of discontinuous groups acting on dual cones. The application of the results in \Cref{subprelim} on hyperbolic lattices as well as a quite powerful way to shorten the calculation time is given in \Cref{secHyp}. In \Cref{secPerformance} we analyse the scope and running time of our algorithm and give some examples. These were calculated using the computer algebra system \textsc{Magma} (cf. \cite{Magma}). The source code for the necessary \textsc{Magma}-package \texttt{AutHyp.m} as well as a short description of the included intrinsics is available via the author's homepage \url{http://www.mi.uni-koeln.de/~mmertens}.

\paragraph*{Notation}\label{parNot}
A \emph{lattice} $(L,\Phi)$ always consists of two data, a free $\Z$-module $L$ of rank $n$ with basis $B=(b_1,\dots ,b_n)$ and a non-degenerate, symmetric bilinear form $\Phi:\V\times \V\rightarrow \R$ where $\V=\R\otimes_\Z L$. The \emph{signature} of the bilinear form will always be given as a pair $(p,-q)$ where $p$ denotes the number of positive and $q$ the number of negative eigenvalues of the \emph{Gram matrix} of $\Phi$ with respect to $B$ which is denoted by ${}_B\Phi^B:=(\Phi(b_i,b_j))_{i,j=1}^n$. By $L^\#$ we denote the \emph{dual lattice} of $L$ and in case that $L$ is integral, i.e. $L\subseteq L^\#$, we write $\Delta(L):=L^\#/L$ for the \emph{discriminant group} of $L$. The \emph{automorphism group} of a lattice $(L,\Phi)$ is defined as 
\[\Aut(L):=\{g\in\GL(\V)\mid Lg=L\text{ and } \Phi(xg,yg)=\Phi(x,y)\text{ for all }x,y\in L\}.\]
If $(L,\Phi)$ is an integral lattice, we can consider $\Aut(L)$ as a subgroup of $\GL_n(\Z)$ by fixing a basis $B$ of $L$:
\[\Aut(L)\cong\Aut_\Z(A):=\{g\in\GL_n(\Z)\mid gAg^{tr} = A\},\]
where $A={}_B\Phi^B$. Note that $\Aut(L)$ acts on $L$ from the right while it acts from the left on the set of Gram matrices of $L$.

For any ring $R$ let $R^n:=R^{1\times n}$ be the free $R$-module of rank $n$ represented as a row vector. By $e_i$ we denote the $i$th row of the $n\times n$ unit matrix $I_n$.

For a subset $S$ of any $R$-module $M$ let $\langle S\rangle_R$ be the submodule of $M$ generated by $S$ (mostly we will omit the subscript $R$ if there are no confusions about the base ring to be worried about). Similarly, if $S$ is a subset of some group $G$, we denote by $\langle S\rangle$ the subgroup of $G$ generated by $S$.
\section{Preliminaries}\label{subprelim}
In this section we briefly recall the most important definitions and results from the first two sections of \cite{Opg01}.

Throughout this section, let $\V$ be a real vector space of dimension $n$ and
\[\sigma:\V\times \V\rightarrow \R\]
be a positive definite bilinear form on $\V$.

Two non-empty open subsets $\Vgt,\Wgt\subset\V$ are called \emph{dual cones} with respect to $\sigma$, if the following properties hold:
\begin{enumerate}[(DC.1)]
\item\label{DC1} For $x\in\Vgt$, $y\in\Wgt$ we have $\sigma(x,y)>0$.
\item\label{DC2} If $\Vigeq$ denotes the topological closure of $\Vigt$ in $\V$ then for any $x\in\V\setminus\Vgt$ there is a $y\in\Wgeq\setminus\{0\}$ such that $\sigma(x,y)\leq 0$. The same holds for changed roles of $x$ and $y$. 
\end{enumerate}

For fixed $\sigma$ and dual cones $\Vgt,\Wgt$ a discrete subset $D\subset\Wgeq\setminus\{0\}$ is called \emph{admissible} if for every $x$ in the boundary $\partial\Vgt$ of $\Vgt$ and every $\eps>0$ there is a $d\in D$ such that $\sigma(x,d)<\eps$ (cf. \cite[Definition 1.4, Lemma 1.5]{Opg01}). For such a set $D$ and some vector $x\in\Vgt$ we define (cf. \cite[Definition 1.2]{Opg01})
\begin{enumerate}[(i)]
\item The $D$\emph{-minimum} of $x$:
\[\mu_D(x):=\min\limits_{d\in D} \sigma(x,d),\]
\item The set of $D$\emph{-minimal vectors} of $x$
\[M_D(x):=\{d\in D\,\mid\, \sigma(x,d)=\mu_D(x)\},\]
\item The $D$\emph{-Voronoi domain} of $x$:
\[\calD_D(x):=\left\lbrace \sum\limits_{d\in M_D(x)} \alpha_dd\,\mid\, \alpha_d>0\right\}.\]
\end{enumerate}

If $\dim\langle M_D(x)\rangle=n$, then we call $x$ a $D$\emph{-perfect} vector and denote the set of all $D$-perfect vectors with $D$-minimum $1$ by $P_D$ (cf. \cite[Definition 1.2]{Opg01}). As shown in \cite[Proposition 1.8]{Opg01}, $P_D$ is not empty because for every $y\in\Vgt$ there exists some $x\in P_D$ such that $\calD_D(y)\subseteq\calD_D(x)$. The proof given there is constructive and can also be used to compute new $D$-perfect points from old ones.

Let $x\in P_D$. A vector $r\in\V_1\setminus\{0\}$ with $\sigma(r,d) \geq 0$ for all $d\in M_D(x)$ and the additional property that $\sigma(r,d)=0$ for $n-1$ linearly independent vectors $d\in M_D(x)$ is called a \emph{direction} of $x$. 

If $r$ is a non-blind direction, i.e. $r\notin\Vgeq$, then there exists a number $\rho>0$ such that $y:=x+\rho r\in P_D$. The $D$-perfect vector $y$ is called a \emph{neighbour} of $x$ in direction $r$, the vectors $x$ and $y$ are called \emph{contiguous}. 

By \cite[Theorem 1.9]{Opg01}, the $D$-Voronoi domains of the $D$-perfect points form a face-to-face tesselation of the whole cone $\Wgt$, i.e. the (undirected) $D$\emph{-Voronoi graph} $\Gamma_D$ with vertex set $P_D$ in which  two vertices are adjacent if and only if the corresponding $D$-perfect points are contiguous, is connected and locally finite.

Now consider a group $\Omega\leq\GL(\V)$ which leaves the cone $\Vgt$ invariant and acts \emph{properly discontinuously} on $\Vgt$

For $\w\in\Omega$ the \emph{adjoint element} of $\w$ is the unique $\w^{ad}\in\GL(\V)$ fulfilling $\sigma(x\w,y)=\sigma(x,y\w^{ad})$ for all $x,y\in\V$. The \emph{adjoint group} of $\Omega$ is defined by
\[\Omega^{ad}:=\{\w^{ad}\,\mid\, \omega\in\Omega\}.\]

By \cite[Lemma 2.1]{Opg01}, the group $\Omega$ acts on the $D$-Voronoi graph $\Gamma_D$ in case that $D$ is admissible.

There is a whole theory built around groups acting on graphs, often called Bass-Serre theory. We refer the reader to the monographs \cite{Dicks} by W. Dicks and \cite{Ser} by J.-P. Serre for the original statement of the Bass-Serre theorem, which implies the following theorem if translated into this context.
\begin{theorem}\label{theoBassSerre}(\cite[Theorem 2.2]{Opg01})\mbox{}\\
Let $\Omega\leq\GL(\V)$ act properly discontinuously on $\Vgt$ and let $D\subset\Wgeq\setminus\{0\}$ discrete, admissible and invariant under the action of $\Omega^{ad}$. Furthermore, let $X$ denote a transversal of $P_D/\Omega$ such that the subgraph $\Gamma_D(X)$ of $\Gamma_D$ generated by $X$ is connected and put $T$ as a maximal spanning tree of $\Gamma_D(X)$. The vertices of $\Gamma_D$ which are adjacent to $T$ are collected in the set $\delta(T)$. Finally, choose for $y\in \delta(T)$ one $\w_y\in\Omega$ such that $y\w_y^{-1}\in X$. Then it holds that
\[\Omega=\langle \w_y,\Stab_\Omega(x) | y\in\delta(T), x\in X\rangle.\]
\end{theorem}

In case that the residue class graph is finite, the above theorem immediately gives a template algorithm to compute a generating system for $\Omega$:
\begin{algo}\label{alg}\mbox{}\\
Input: $\Vgt,\,\Wgt$ dual cones w.r.t. $\sigma$, $D\subset\Wgeq\setminus\{0\}$ discrete, admissible and $\Omega$-invariant.

\noindent Output: a transversal for $P_D/\Omega$, generators for $\Omega$.
\begin{enumerate}
\item Find a $D$-perfect point $x_1$ and initalize $L_1=\{x_1\},\,L_2:=\emptyset,\,S:=\emptyset$.
\item If $L_1=\emptyset$, then return $L_2,\,S$; else choose $x\in L_1$.
\item Compute a generating system $S_x$ for $\Omega_x$ and put $S:=S\cup S_x$.
\item Compute the set $R$ of neighbours of $x$ and a transversal $R'$ of $R/\Omega_x$.
\item For $y\in R'$ decide whether there is a $z\in L_1\cup L_2$ and an $\w\in\Omega$ with $y\omega=z$.

If not, put $L_1:=L_1\cup\{y\}$.

If $z\in L_1$ then put $S:=S\cup\{\w\}$.
\item Put $L_2:=L_2\cup \{x\}$, $L_1:=L_1\setminus\{x\}$. Go to step $2$.
\end{enumerate}
\end{algo}
One of the crucial problems using the above method is to prove the finiteness of the residue class graph $\Gamma_D/\Omega$. For Opgenorth's normalizer algorithm this was proven in \cite{Jaq} and \cite{Opg96}. In general there is the following theorem which is a generalization of \cite[Satz 6]{Koe60}.
\begin{theorem}\label{theofundfin}(Koecher, 1960)\\
Let $L$ be a full lattice in $\V$ and $D\subseteq \Wgeq\cap L\setminus \{0\}$ admissible. Assume that $L\Omega^{ad}=L$ and $D\Omega^{ad}=D$. Then the following are equivalent:
\begin{enumerate}[(1)]
\item There are only finitely many $\Omega$-equivalence classes of $D$-perfect points with $D$-minimum $1$.
\item There exists a set $\tilde\FF\subset\Wgt$ fulfilling that for every $y\in\Wgt$ there is an $\w\in\Omega$ with $x\w\in\tilde\FF$, such that there is a $y_0\in\Wgt$ with $d-y_0\in\Wgeq$ for all $d\in\tilde\FF\cap L$.
\end{enumerate}
\end{theorem}
\begin{proof}[Sketch of a proof]
\underline{$(1)\Rightarrow (2)$:} This is the statement of \cite[Satz 6]{Koe60} which Koecher stated for self-dual cones. \emph{Mutando mutandis} the same arguments work for dual cones as well.

Let 
\[\FF(D,\Omega):=\bigcup\limits_{x\in P_D/\Omega}\calD_D(x)/\Omega_x^{ad}.\]
If $P_D/\Omega$ is finite, then $\FF:=\FF(D,\Omega)\cap\Wgt$ is a fundamental domain for the action of $\Omega^{ad}$ on $\Wgt$ (cf \cite[Satz 4]{Koe60}). Then one checks that choosing $\tilde\FF=\FF$ fulfills $(2)$. In order to do that, one remarks that since $\FF$ is contained in the union of finitely many $D$-Voronoi domains that it suffices to show the following assertion.

\begin{quote} 
For all $x\in P_D$ there is a $y_0=y_0(x)\in\Wgt$ such that for all \mbox{$d\in \calD_D(x)\cap\Wgt\cap L$} we have that $d-y_0\in\Wgeq$.
\end{quote}

Since $\calD_D(x)$ is a finite union of pyramides $P(M)$ generated by a set of precisely $n=\dim\V$ vectors, thus we can even restrict ourselves to proving the above claim for this kind of pyramides. Then it is easy to see that for 
\[y_\ell:=\frac 1\gamma \sum\limits_{\substack{k=1\\ \lambda_k\neq 0}}^nv_k\in\Wgt,\]
where $M=\{v_1,\dots ,v_n\}$ and $\gamma:=\vert L/\langle M\rangle_\Z\vert$, we have $\ell-y_\ell\in\Wgeq$ for all $\ell\in L\cap P(M)\cap \Wgt$. Since there can only be finitely many such vectors $y_\ell$, the claim follows.

\underline{$(2)\Rightarrow (1)$} Define $D_0:=\tilde\FF\cap D\subseteq\tilde\FF\cap L$. By assumption we have that for any $d\in D$ there is a $\omega\in\Omega$ such that $d\w^{ad}\in D_0$ and that there is a $y_0\in\Wgt$ with $d-y_0\in\Wgeq$ for all $d\in D_0$. This already implies the finiteness of the residue class graph since there are but finitely many $x\in P_D$ such that $M_D(x)\cap D_0\neq \emptyset$. This follows from the fact that $\sigma(x,x)$ is bounded from above (\cite[p. 400]{Koe60}, \cite[Lemma 3.3.2]{Mert12}) and that $P_D$ is discrete (\cite[Lemma 1.6]{Opg01}).
\end{proof}
The condition $(2)$ in Theorem\nobreakspace \ref {theofundfin} is fulfilled for example if there is a finite set $M\subset D$ such that the pyramide $P(M)=\{\sum_{d\in M}\alpha_dd |\alpha_d\geq 0\}$ contains a fundamental domain for the action of $\Omega^{ad}$ (cf. \cite[Proposition 2.5]{Opg01}).

Geometrically this condition $(2)$ means that you can shift the pseudo fundamental domain $\tilde\FF$ a little bit towards the origin and the lattice points contained in it stay in $\Wgeq$. In particular this implies that $\tilde\FF$ can be embedded in a pyramide $P(M)$ such that at most finitely many of its rays lie in the cone's boundary.
\section{Hyperbolic Lattices}\label{secHyp}
In this section we are going to give detailed and explicit methods how to apply \Cref{alg} to compute generators for the automorphism group of a hyperbolic lattice. For this purpose, we first have to find a pair of dual cones with respect to a bilinear form.

Throughout this section, we consider integral lattices $(L,\Phi)$ where $\Phi$ is a bilinear form of signature $(n-1,-1)$. We also assume that there is a fixed basis $B$ of $L$ and we set $A:={}_B\Phi^B\in\Z^{n\times n}$. The set of real (integral) hyperbolic matrices is denoted by $\Rhyp$ ($\Zhyp$).

With respect to this basis, we identify $\V:=\R\otimes_\Z L\cong\R^n$. Therefore we will not always distinguish strictly between the lattice $(L,\Phi)$ and its Gram matrix $A$. 
\subsection{Hyperbolic Lattices and Dual Cones}
\begin{lemma}\mbox{}\\
Consider the standard hyperbolic matrix 
\[H:=diag(-1,1,\dots 1)\in\Zhyp.\]
Then the set 
\begin{align}
\mathcal{C}:=\{(x_1,\dots ,x_n)\in\R^n\,\mid\, xHx^{tr}<0\text{ and }x_1>0\}\label{eqC}
\end{align}
is a selfdual cone with respect to the standard scalar product.
\end{lemma}
\begin{proof}
Obviously, $\mathcal{C}$ is an open, nonempty set. 

To prove the property (DC.1), let $x=(x_1,\dots ,x_n),y=(y_1,\dots ,y_n)\in\calC$. Without loss of generality, we may assume that $x_1=y_1$, since we can rescale $y$ by a positive real number. We calculate
{\allowdisplaybreaks
\begin{align*}
xy^{tr}&=x_1y_1+\sum\limits_{i=2}^n x_iy_i\\
       &=\frac{1}{2}(x_1^2+y_1^2)+\sum\limits_{i=2}^n \frac{1}{2}((x_i+y_i)^2-x_1^2-y_i^2)\\
       &=\frac{1}{2}\left( (x_1^2-\sum\limits_{i=2}^n x_i^2)   +(y_1^2-\sum\limits_{i=2}^n y_i^2)+\sum\limits_{i=2}^n(x_i-y_i)^2 \right)\\
       &=\frac{1}{2}\left(\underbrace{-xHx^{tr}}_{>0}+\underbrace{(-yHy^{tr})}_{>0}+\underbrace{\sum\limits_{i=2}^n(x_i-y_i)^2 }_{\geq 0}\right)>0
\end{align*}
}
hence (DC.1) is shown.

Regarding (DC.2) let $x=(x_1,\dots ,x_n)\not\in\calC$. If $x_1\leq 0$, then $y=(1,0,\dots ,0)\in\calC$ fulfills $xy^{tr}=x_1\leq 0$. Thus suppose $x_1>0$. By definition of $\calC$ we know that $-x_1^2+ s\geq 0$ for $s:=\sum_{i=2}^n x_i^2$, hence especially $s>0$. As above we may therefore assume $s=1$. This yields $0< x_1^2\leq 1$ . Now define $y=(1,-x_2,\dots , -x_n)$. It holds $yHy^{tr}=0$ and $y_1=1>0$, thus $y\in\overline\calC$ and $xy^{tr}=x_1-1\leq 0$ as desired.
\end{proof}
We remark here that for dual cones $\Vgt,\Wgt$ with respect to a bilinear form $\sigma$ the sets $\Vgt g$ and $\Wgt (g^{ad})^{-1}$ are dual cones with respect to $\sigma$ for every $g\in\GL(\V)$.

Now Sylvester's Law of Inertia states that for any $A\in\Rhyp$ there is a $T\in\GL_n(\R)$ such that $TAT^{tr}=H$. Thus we obtain that for any such $T$
\begin{align}
\Vgt&:=\calC T^{-1}=\{x\in\R^n\,\mid\, xAx^{tr}<0 \text{ and }xy_1^{tr}>0\}\label{eqcones1}\\
\Wgt&:=\calC T^{tr}=\{x\in\R^n\,\mid\, xA^{-1}x^{tr}<0 \text{ and }xy_2^{tr}>0\}\label{eqcones2},
\end{align}
are dual cones with respect to the standard scalar product, where $y_1=e_1T^{-1}$ and $y_2=e_1T^{tr}$.
\begin{remark}\mbox{}\\
Note that the dual cones $\Vgt$ and $\Wgt$ in \eqref{eqcones1} and \eqref{eqcones2} depend on the choice of $T$, but the second output of \Cref{alg} being a generating system for the group $\Omega$ does not depend on the choice of the dual cones, as long as $\Omega$ acts on them. 

In fact $\Omega=\Aut_\Z(A)$ does \emph{not} act on $\Vgt$ because one has $-I_n\in\Aut_\Z(A)$ for all symmetric matrices $A$ but certainly $\Vgt\neq\Vgt (-I_n)$. But $\Aut_\Z(A)$ acts on $\Vgt\cup(-\Vgt)$ and thus we can restrict ourselves to the stabilizer of $\Vgt$ in $\Aut_\Z(A)$ which from now on will be denoted by $\Omega$. We shall call $\Omega$ the \emph{reduced} automorphism group of $A$ or $L$ respectively. Obviously we have $\Aut_\Z(A)=\Omega\times\langle -I_n\rangle$.
\end{remark}
\begin{lemma}\mbox{}\\
The group $\Omega$ acts properly discontinuously on $\Vgt$.
\end{lemma}
\begin{proof}
Suppose that there is a cluster point $x^*$ of the orbit $x\Omega$ of some $x\in\Vgt$. Then there is a sequence $(\w_k)_{k\in\N}$ in $\Omega$ such that $x\omega_k$ tends to $x^*$ as $k$ tends to $\infty$. Obviously we have $xAx^{tr}=(x\w_k)A(x\w_k)^{tr}$ for all $k$, hence also $x^*A{x^*}^{tr}=xAx^{tr}$. Since the real automorphism group modulo $-I_n$ denoted $\Omega_\R$ acts transitively on the set $\{x\in\Vgt\,\mid\,xAx^{tr}=c\}$ for any given number $c>0$, there must be an $\omega^*\in\Omega_\R$ such that $x\w^*=x^*$. By choosing the right representatives modulo $\Stab_{\Omega_\R}(x^*)$ we may assume that $\omega_k\rightarrow \omega^*$, but because $\Omega$ is discrete this means that the sequence $(\w_k)_{k\in\N}$ and therefore the sequence $(x\w_k)_{k\in \N}$ becomes constant at some point.

As will be shown in \Cref{lemStab}, the stabilizer of any $x\in\Vgt$ is isomorphic to a subgroup of the automorphism group of a positive definite lattice, thus it must be finite.

Therefore the action of $\Omega$ on $\Vgt$ is properly discontinuous. 
\end{proof}
\subsection{Admissibility}
For the rest of this paper let
\begin{align} 
D:=\Z^n\cap\Wgeq\setminus\{0\}.
\label{eqD}
\end{align}
This set is obviously discrete and $\Omega^{ad}$-invariant. In this subsection we want to prove that $D$ is admissible, which is less obvious.
\begin{remark}\label{corMink}\mbox{}\\
For every $\eps>0$ and $x\in\R^n$ the set
\[P_\eps(x):=\{y\in\R^n\,\mid\, 0\leq xy^{tr} \leq \eps\}\]
contains a point $\ell\neq 0$ with integral coordinates.
\end{remark}
\begin{proof}
For given $\eps>0$ and $x\in\R^n$ consider
\[X:=\{y\in\R^n\,\mid\,-\eps\leq xy^{tr}\leq\eps\}.\]
This set is clearly centrally symmetric and convex and has infinite volume. Thus by Minkowski's Convex Body Theorem (\cite[Satz 4.4]{Neu}) we know that $X\cap \Z^n\neq\{0\}$. For $\ell\in X\cap\Z^n\setminus\{0\}$ we either have $x\ell^{tr}\geq 0$ and thus $\ell\in P_\eps(x)$ or $x\ell^{tr}\leq 0$ and thus $-\ell\in P_\eps(x)$.
\end{proof}
In addition we use a famous result by Lov\'{a}sz. To understand it we need the following definition (cf. \cite{Lov}).
\begin{definition}\mbox{}
\begin{enumerate}[(i)]
\item Let $S\subseteq\R^n$ convex. We call $S$ a \emph{maximal lattice-free convex set}, for short \emph{MLFC-set}, if it holds that
\begin{enumerate}
\item\label{p1} the relative interior of $S$ with respect to $\langle S\rangle_\R$ does not contain points with integral coordinates,
\item $S$ is maximal (with respect to inclusion) among all convex subsets $T\subseteq\R^n$ fulfilling property (a).
\end{enumerate}
\item Let $\mathcal{U}$ be an affine subspace of $\R^n$. If the integral points in $\mathcal{U}$ generate $\mathcal{U}$ as an affine space, i.e.
\[\langle \mathcal{U}\cap\Z^n\rangle_{aff}=\mathcal{U},\]
then we call $\mathcal{U}$ a \emph{rational subspace} of $\R^n$. If not, $\mathcal{U}$ is called \emph{irrational}.
\end{enumerate}
\end{definition}
\begin{theorem}\label{theoLovasz}(Lov\'asz, 1989)\\
Let $S\subseteq\R^n$. It holds that $S$ is an MCLF-set if and only if one of the following conditions hold:
\begin{enumerate}
\item\label{L1} $S$ is a convex polyhedron of the form $S=P+\calL$, where $P$ is a polytope and $\calL$ is a rational vector space, whereby it holds that $\dim\langle P\rangle+\dim\calL=n$ and $S^\circ\cap\Z^n=\emptyset$ and every surface of $S$ contains an integral point in its relative interior.
\item $S$ is an irrational affine hyperplane of $\R^n$ of dimension $n-1$.
\end{enumerate}
\end{theorem}
A proof of this can be found in \cite{Lov}.
\begin{remark}\mbox{}\\
Let $S=P+\calL$ an MCLF-set as in \Cref{theoLovasz},\ref{L1}. Then $P$ must be a bounded polytope since the surfaces of $S$ contain an integral point in their relative interior. By adding suiting integral points if necessary one gets a basis of a full-rank sublattice of $\Z^n$ which has a fundamental polytope $F$. The projection of $F$ onto $\langle P\rangle$ must be contained properly in $P$ if $P$ is unbounded which yields a contradiction to $S$ being lattice-free.  
\end{remark}
We can now prove the following:
\begin{proposition}\label{propDadm}\mbox{}\\
The set $D$ from equation \eqref{eqD} is admissible.
\end{proposition}
\begin{proof}
Let $x\in\partial\Vgt$. By \cite[Lemma 1.1]{Opg01} there is a $y\in\Wgeq\setminus\{0\}$ such that $xy^{tr}=0$. Now let $\eps >0$ and define
\[M_\eps:=\{v\in\Wgeq\setminus\{0\}\,\mid\, xv^{tr}\leq\eps\}.\]
We have to show that $M_\eps\cap D\neq\emptyset$.

Assume on the contrary that this were false. Then, since $M_\eps$ is convex, there must be an MCLF-set $S\subset\R^n$ with $M_\eps\subseteq S$. Obviously, $S$ cannot be an affine hyperplane, hence $S=P+\calL$ as in \Cref{theoLovasz},\ref{L1}. Since $P$ is bounded, it can't hold that $y\in P$, thus $y=p+\ell$ for some $p\in P$, $\ell\in\calL\setminus\{0\}$. For reasons of dimension, $p$ and $\ell$ are in fact unique. Now let $\lambda>0$. Thus $\lambda y=\lambda p+\lambda \ell$, but since $P$ is bounded and $\lambda y\in M_\eps\subseteq S=P+\calL$, this cannot hold for big values of $\lambda$ if $p\neq 0$. Thus $y\in\calL$.

Now consider for $0\leq\eps'\leq\eps$ the set
\[T_{\eps'}(x):=\{v\in\R^n\,\mid\, xv^{tr}=\eps'\}.\]
Then for each such $\eps'$ we can find some $v_{\eps'}\in\Wgeq\cap T_{\eps'}(x)$, hence we can write
\[T_{\eps'}(x)=v_{\eps'}+T_0(x)\] 
and thus 
\[P_\eps=\bigcup\limits_{0\leq\eps'\leq\eps} (v_{\eps'}+T_0(x)).\]
For fixed $0\leq \eps'\leq \eps$ let $v_{\eps'}=p'+\ell'$ where $p\in P$ and $\ell\in\calL$. Now for every $y\in T_0(x)$, there is some $\lambda_0> 0$ such that
\[v_{\eps'}+\lambda y\in\Wgeq\text{ for all }\lambda\geq\lambda_0.\]
We can decompose $v_{\eps'}+\lambda_0 y=p+\ell$ with $p\in P$, $\ell\in\calL$, thus we have 
\[\lambda_0 y=(p-p')+(\ell-\ell').\]
For $\lambda\geq\lambda_0$ this implies that
\[v_{\eps'}+\lambda y=\underbrace{p'+\frac{\lambda}{\lambda_0} (p-p')}_{\in P}+\underbrace{\ell'+\frac{\lambda}{\lambda_0}(\ell-\ell')}_{\in\calL}\]
which can only hold if $p=p'$ because $P$ is bounded. But this means that $y\in\calL$ and hence it follows $v_{\eps'}+T_0(x)\subseteq P+\calL$ for all $\eps'$, and therefore $P_\eps(x)\subseteq S$ which is a contradiction to \Cref{corMink}.
\end{proof}
\subsection{$D$-Minimal Vectors and Automorphisms}
Now that we have found the dual cones and an admissible set $D$, we give explicit methods to calculate the $D$-minimal vectors of a point $x\in\Vgt$ as well as its stabilizer and connecting elements. 

As stated before, we consider $A\in\Zhyp$ a given hyperbolic matrix with a fixed pair of dual cones $\Vgt,\Wgt$ and cone test vectors $y_1,y_2$ as in equations \eqref{eqcones1} and \eqref{eqcones2} and $D$ as in equation \eqref{eqD}.
\begin{remark}\mbox{}
\begin{enumerate}[(i)]
\item For any $x\in\Vgt$, it holds that $-xA\in\Wgt$.
\item Let $x\in\Vgt\cap\Z^n$. Then we have $\mu_D(x)\leq -xAx^{tr}=:N(x)$.
\end{enumerate}
\end{remark}
\begin{proof}
$(i)$ is clear for $A=H$ with $H=diag(-1,1,\dots ,1)$ and in general follows from Sylvester's Law of Inertia.

$(ii)$ follows from $(i)$ because $-xA\in D$ and thus $\mu_D(x)\leq x(-xA)^{tr}=-xAx^{tr}$.
\end{proof}
\begin{lemma} \mbox{}\\
Let $x\in\Vgt\cap\Z^n$ and $\lambda>0$ and define the affine hyperplane
\[H_\lambda(x):=\{-\lambda xA+y\in\R^n\,\mid\,xy^{tr}=0\}.\]
If $\lambda$ is minimal such that the finite set
\[M_\lambda=\{d\in H_\lambda (x)\cap\Z^n\,\mid\, (d+\lambda xA)A^{-1}(d+\lambda xA)^{tr}\leq \lambda^2xAx^{tr}\text{ and }y_2d^{tr}\geq 0\}\]
is nonempty, it holds that
\[M_\lambda=M_D(x).\]
\end{lemma}
\begin{proof}
First of all we note that there actually is a minimal $\lambda$ such that $M_\lambda\neq\emptyset$. For that let $d=-\lambda xA+y\in H_\lambda(x)\cap\Z^n$. Since by assumption $x\in\Z^n$ and $A\in\Zhyp$ it holds that
\[xd^{tr}=-\lambda \underbrace{xAx^{tr}}_{\in\Z}+\underbrace{xy^{tr}}_{=0}=\lambda N(x)\in\Z,\]
thus it follows that for $M_\lambda$ to be nonempty it is necessary that $\lambda\in\tfrac{1}{N(x)}\Z$. 

We calculate 
\[dA^{-1}d^{tr}=-\lambda^2N(x)+yA^{-1}y^{tr},\]
thus minimizing $\lambda$ means minimizing $xd^{tr}$.

For $d$ to belong to $\Wgeq$ it is further necessary that
\[(d+\lambda xA)A^{-1}(d+\lambda xA)^{tr}\leq \lambda^2 N(x),\]
and thus we get that $M_\lambda\subseteq D$, and hence $M_\lambda=M_D(x)$ for the minimal $\lambda$ with $M_\lambda\neq\emptyset$.
\end{proof}
\begin{remark}\mbox{}\\
There are algorithms to calculate short vectors in positive definite lattices, see for instance \cite[pp. 188-190]{PoZ}. We can use them to calculate the set $M_\lambda$ for $\lambda=\tfrac{1}{N(x)},\tfrac{2}{N(x)},\dots$ until $M_\lambda$ is not empty in the following way: 

The matrix $A^{-1}$ induces a positive definite scalar product on the subspace $H_0(x)$ of $\R^n$. We can thus calculate all $y\in H_0(x)$ such that $d=-\lambda xA +y\in\Z^n$, which especially means that $y$ is contained in the lattice $\tfrac{1}{N(x)}\Z^n\cap H_0(x)$ endowed with the bilinear form induced by $A^{-1}$, and 
\[yA^{-1}y^{tr}\leq \lambda^2N(x).\]

Since $-xA\in H_1(x)$ is an integral point, we have to repeat this at most $N(x)$ times.
\end{remark}
\begin{lemma}\label{lemStab}\mbox{}
\begin{enumerate}[(i)]
\item Since $A$ is an integral matrix, all $D$-perfect vectors have integer entries (up to scaling). Thus we can restrict ourselves on integral vectors $x$.
\item Let $x$ an integral $D$-perfect point and $g\in S_x:=\Stab_\Omega(x)$. Then $g$ induces an automorphism of the positive definite lattice $L(x)=H_0(xA)\cap\Z^n$ with scalar product induced by $A$.
\end{enumerate}
\end{lemma}
\begin{proof}
$(i)$ is clear. 

Regarding $(ii)$, we have that $x\in\Vgt$ and thus $xAx^{tr}<0$. In addition it is perpendicular to $L(x)$ with respect to the bilinear form induced by $A$, hence this bilinear form must be positive definite on the orthogonal complement $H_0(xA)$ of $x$. That $g$ induces an automorphism of $L(x)$ is easily calculated.
\end{proof}
\begin{remark}\label{remStab}\mbox{}
\begin{enumerate}[(i)]
\item There are quite powerful methods to calculate the automorphism group of a positive definite lattice, see for instance \cite{PlS}. Thus we can calculate the automorphisms of $L(x)$ and continue them such that they act trivially on $x$. This can be done by adding $x$ to a lattice basis of $L(x)$ to obtain a basis of $\R^n$. With respect to this basis, the automorphisms have the form
\[\left(\begin{array}{c|c} g & 0 \\ \hline 0 & 1\end{array}\right),\quad g\in \Aut(L(x)).\]
Changing the basis to $B$ chosen in the very beginning yields a subgroup $G_x$ of rational automorphisms of $A$. The elements in $G_x$ that fix $\Z^n$ as a point set are exactly the elements in $S_x$.
\item A slight modification of this also gives a method to determine the connecting elements in \Cref{theoBassSerre}. For $D$-perfect vectors $x,y\in\Vgt$ with $N(x)=N(y)$ we calculate integral isometries of the lattices $L(x)\rightarrow L(y)$. If there are no isometries, then $x$ and $y$ are inequivalent under $\Omega$. If there are, then continuing them onto all of $\R^n$ such that $x\mapsto y$ yields again some rational automorphisms of $A$, of which it is to determine, whether they are integral or not. If there is an integral one, this yields such a connecting element denoted by $\omega_y$ in \Cref{theoBassSerre}. If not, then again, $x$ and $y$ are not equivalent under $\Omega$.

Note that this last case does in fact occur.
\end{enumerate}
\end{remark}
\subsection{Finiteness of the Residue Class Graph}
Up until now, we have given explicit methods to solve all the computational tasks in the algorithm in \cref{subprelim}. But it is not at all clear that \Cref{alg} terminates for any integral hyperbolic lattice, i.e. for the reduced automorphism group of any such lattice there is only a finite number of inequivalent perfect points. In this subsection we will prove this using similar methods as in \cite{Siegel}, where C.L. Siegel has proven that the automorphism group of any definite or indefinite lattice is finitely generated.

We recall some of Siegel's results. For that let $S\in\Z^{n\times n}$ be an integral symmetric matrix of signature $(m, -(n-m))$ and $\Omega$ its reduced automorphism group. Consider the equation
\begin{align}\label{HSH}
HS^{-1}H=S
\end{align}
for positive definite matrices $H\in\Rsgt$. The solutions of equation \eqref{HSH} form a rational manifold $\mathfrak{H}$ of dimension $m(n-m)$ which can be parametrized by
\begin{align}\label{Hmanifold}
H=2Z-S,\quad Z=\left(\begin{array}{c|c} T^{-1} & T^{-1}Y^{tr} \\  \hline YT^{-1} & YT^{-1}Y^{tr} \end{array}\right),\quad T=(I_m\mid Y^{tr})S^{-1}\left(\begin{array}{c} I_m \\ \hline Y\end{array}\right),
\end{align}
where where $Y\in\R^{(n-m)\times m}$ is chosen such that $T$ is positive definite (see \cite[Equation (35)]{Siegel}). Note that $\Omega$ acts from the left on $\mathfrak{H}$ by 
\[(g,H)\mapsto gHg^{tr}\]
and from the right on the corresponding set of $Y$ by
\[(Y,g)\mapsto g^{tr}\left(\begin{array}{c} I_m \\ \hline Y\end{array}\right)\]
where we have to renormalize the last term (by multiplication with a matrix $Q\in\GL_m(\R)$) such that the upper part becomes again $I_m$. Thus we get a rational injection
\begin{align}\label{eqphi}
\phi:\mathfrak{H}\rightarrow \R^{(n-m)\times m},\, H\mapsto Y
\end{align}
which is $\Omega$-equivariant. 

Using reduction theory of quadratic forms due to H. Minkowski and results by C. Hermite, Siegel constructs a fundamental domain $\FF$ of $\Omega$ in $\FH$.

\begin{theorem}\label{theoSiegel}(Siegel, 1940)\\
Let $S\in\Z^{n\times n}$ be an integral symmetric matrix of signature $(m,-(n-m))$.
There is a finite number of hyperplanes in the subspace of $\Rsym$ spanned by $\FH$ such that $\FF$ is an intersection of the corresponding halfspaces and the manifold $\FH$ of all matrices $H$ fulfilling equation \eqref{HSH}. In addition, $\FF$ has only finitely many neighbours, that means there are only finitely many elements $\w\in\Omega$ with $\FF\cap\FF\w\neq\emptyset$.
\end{theorem}
\begin{proof}
\cite[Satz 10]{Siegel}
\end{proof}
Equipped with these results we can now prove the main result of this subsection.
\begin{theorem}
Let $A\in\Zhyp$ be an integral hyperbolic matrix, $\Omega$ its reduced automorphism group and $\Vigt$ and $D$ as in equations \eqref{eqcones1}-\eqref{eqD}. Then the residue class graph $\Gamma_D/\Omega$ is finite. In particular, \Cref{alg} terminates.
\end{theorem}
\begin{proof}
Obviously, we can replace $A$ by $-A$ without changing $\Omega$. Especially it holds that $x(-A)^{-1}x^{tr}>0$ for every $x\in\Wgt$. The signature of $-A$ is $(1,1-n)$, thus the manifold $\FH$ as described in equation \eqref{Hmanifold} is homeomorphic via $\phi$ in equation \eqref{eqphi} to the manifold $\Wgt/\R_{>0}$ as we can rescale each cone vector such that its first coordinate is $1$. Since $\phi$ is $\Omega$-equivariant, the fundamental domain $\FF$ is mapped to a fundamental domain $\tilde\FF\subset\Wgt$ of the action of $\Omega$. \Cref{theoSiegel} tells us that $\FF$ and therefore $\tilde\FF$ are bounded by finitely many hyperplanes in $\R^n$. But then \Cref{theofundfin} implies that there are only finitely many $D$-perfect points modulo $\Omega$, hence the assertion follows.
\end{proof}
\begin{remark}
Since \Cref{theoSiegel} is not restricted to a certain signature of the matrix $S$, it is very likely that the method described here will yield a more general procedure to determine the automorphism group of any indefinite lattice. In that case one would have to define the dual cones as subsets of $\R^{n\times m}$ in an analogous way, replacing the symbols $<$ and $>$ by "negative" and "positive definite" respectively.  
\end{remark}
\subsection{The Watson Process}\label{subWatson}
In this subsection we are going to recall the Watson process which enables us to reduce the discriminant of a quadratic form without increasing its class number, see e.g. \cite{Watson} and \cite{Lorch} for details. For that purpose, consider for a moment $(L,\Phi)$ to be an integral lattice with a symmetric, non-degenerate bilinear form $\Phi$ of arbitrary signature.
\begin{definition}\mbox{}\\
Let $p$ a prime number such that $\Delta(L)$ contains an element of order $p^2$. The integral lattice
\[Fill_p(L):=(pL^\#\cap L, \frac{1}{p}\Phi)\]
is called a $p$-\emph{filling} of $L$.
\end{definition}
The $p$-filling has the following effect on the genus symbol of $L$ in the sense of \cite[Chapter 15]{CS}: For the $p$-filling to be defined, the highest level of a $p$-adic Jordan constituent of $L$ must be at least $2$. The $p$-filling reduces this level by $2$. Repeating this until it is no longer possible yields the so called \emph{Watson lattice} of $L$, denoted by $\Watson(L)$. The mapping $L\mapsto \Watson(L)$ is called the \emph{Watson process}.
\begin{remark}\mbox{}\\
For an integral lattice $L$, the discriminant group $\Delta(\Watson(L))$ has squarefree exponent by construction and, which is even more important for computational issues, the discriminant $\vert\Delta(L)\vert$ is decreased by the Watson process by a factor $p^2$ for every $p$-filling during the Watson process.
\end{remark}
\begin{remark}\mbox{}\\
Since $p$-filling (and therefore the Watson process) cannot be expressed as an action of $\GL_n(\Z)$, the automorphism groups of $L$ and $Fill_p(L)$ will not be isomorphic. But up to isomorphism one always has
\[\Aut(L)\leq\Aut(Fill_p(L))\quad\text{and}\quad [\Aut(Fill_p(L)):\Aut(L)]<\infty.\]
Hence one can determine $\Aut(L)$ again via an orbit-stabilizer computation in $\Aut(Fill_p(L))$ or $\Aut(\Watson(L))$ respectively.
\end{remark}
Although we cannot prove it, so far all examples suggest that first calculating $\Aut_\Z(\Watson(A))$\footnote{We use the same notation for the Gram matrices of the lattices involved} via the algorithm and then finding $\Aut_\Z(A)$ as a co-finite subgroup in it via an orbit-stabilizer calculation makes the computation very much faster, provided that the Watson process has any effect.

Heuristically, this could be explained by two observations: Since the automorphism group of $\Watson(L)$ properly contains the one of $L$ up to isomorphism, the corresponding equivalence relation on the set of $D$-perfect points is in a way coarser for $\Watson(L)$ than for $L$. Thus one can expect that modulo $\Aut(\Watson(L))$ there will be fewer classes of $D$-perfect points than modulo $\Aut(L)$. 

Furthermore, the computation time to determine $D$-minimal vectors of $x\in\Vgt\cap\Z^n$ depends among others on the discriminant of the lattice $(\tfrac{1}{N(x)}\Z^n\cap H_0(x),A^{-1})$ and the number $N(x)$ (cf. \Cref{subRun}). Both numbers tend to be smaller for $\Watson(L)$ than for $L$. Since in most cases one has to calculate $D$-minimal vectors for many different vectors $x$, an acceleration of the computation of $D$-minimal vectors is quite valuable.

The strength of this slight modification is illustrated in \Cref{exWatson} and \Cref{subStat}.
\section{Performance of the algorithm}\label{secPerformance}
\subsection{Running time}\label{subRun}
It seems rather hard to give an a priori estimate of the running time of our algorithm, but as an a posteriori estimate we can give the following: Suppose $A\in\Zhyp$ has $d$ $D$-perfect points $x_1,\dots ,x_d$ where $x_i$ has $M_i$ $D$-minimal vectors and $r_i$ inequivalent directions. Since by experiments we see that our algorithm is only practicable in low dimensions we assume that computing short vectors and isometries of positive definite lattices as well as performing the Watson process takes constant time and that arithmetical operations like matrix-vector multiplication or calculating the content of an integral matrix don't consume any time at all. The crucial steps in the algorithm are thus the computation of
\begin{enumerate}
\item $D$-minimal vectors and $D$-short vectors ($M_D^C(x):=\{d\in D\,\mid \, (x,d)\leq C\}$ for given $C>0$) of a point $x$,
\item neighbours of a point,
\item directions of a point.
\end{enumerate}
In the author's \textsc{Magma} implementation of the algorithm, the calculation of $M_D(x)$ (resp. $M_D^C(x)$) mainly depends on the value $N(x)=-xAx^{tr}$ (resp. on $C$ and $N(x)$) since one basically computes short vectors in the positive definite orthogonal complement of $x$ this many times (in the worst case).  

For the algorithm to compute neighbours of a point, one does not have such an estimate because one can only say that this procedure finds a neighbour in a given direction in a finite amount of time but not necessarily in a bounded time. The reason for this is that in the implementation a sort of bisection method is used. Usually, this doesn't have to be repeated very often, but in each repetition, the $D$-minimal vectors of a point with increasingly bigger norm $-yAy^{tr}$ have to be computed.

The calculation for finding the directions of a point is carried out using a built-in algorithm of \textsc{Magma}, the running time of which depends exponentially on the dimension $n$ and polynomially on the number of $D$-minimal vectors of the point (cf. \cite{Chazelle}).

Roughly, one has to compute $\sum r_i$ neighbours of points, at least $d$ times the directions of (perfect) points, and $D$-minimal vectors at least $d+2\sum r_i$ times. The higher the dimension gets, the less probable it becomes that the first point one picks in the cone $\V_1$ is (close to) perfect, thus one has to repeat these computations even more often.  

Besides the running time, memory can quickly become a major issue. Most calculated examples needed several Gigabytes of memory.
\subsection{Examples}\label{secEx}
All computations were executed on a \texttt{Quad-Core AMD Opteron(tm) Processor 8356} running at 1150 GHz. The used \textsc{Magma} version is V2.19-2.

Throughout this section we consider the lattice $(\Z^n,\Phi)$ where the Gram matrix of $\Phi$ with respect to the standard basis is given by $A$.
\begin{example}\label{exComplete}\mbox{}\\
Our first example we discuss the steps of the algorithm in detail. We take here
\[A=\begin{pmatrix}
-1 & -3 & -1 \\
-3 & 14 & 8 \\
-1 & 8 & 11
\end{pmatrix},\quad A^{-1}=\frac{1}{155}\begin{pmatrix}
-90 & -25 & 10 \\
-25 & 12 & -11 \\
10 & -11 & 23
\end{pmatrix}.
\]
For this matrix the algorithm finds $9$ inequivalent $D$-perfect points. Rescaled such that their entries become integral, they are represented by
\begin{align*}
x_1=&  \begin{pmatrix}
1 & 0 & 0
\end{pmatrix}
   & x_2= & \begin{pmatrix}
2 & 1 & -1
\end{pmatrix}
  & x_3= &  \begin{pmatrix}
2 & 1 & 0
\end{pmatrix}
 \\
x_4=&  \begin{pmatrix}
9 & 0 & -2
\end{pmatrix}
  & x_5= & \begin{pmatrix}
5 & 3 & -3
\end{pmatrix}
  & x_6= &  \begin{pmatrix}
12 & 5 & -7
\end{pmatrix}
 \\
x_7=&  \begin{pmatrix}
3 & 2 & -1
\end{pmatrix}
  & x_8= &  \begin{pmatrix}
14 & 9 & -2
\end{pmatrix}
 & x_9= &  \begin{pmatrix}
21 & 8 & -12
\end{pmatrix}
\end{align*}
where these points each have got (in the listed order) $8,\,
4,\,
6,\,
8,\,
4,\,
3,\,
4,\,
3$ and $6$ neighbours.
The stabilizers of $x_2$, $x_5$, $x_6$, $x_7$, $x_8$ are trivial and all the other stabilizers are cyclic of order $2$ given by
\begin{align*}
\Stab(x_1)=& \langle \begin{pmatrix}
1 & 0 & 0 \\
6 & -1 & 0 \\
2 & 0 & -1
\end{pmatrix}
\rangle & \Stab(x_3)=&\langle \begin{pmatrix}
9 & 5 & 0 \\
-16 & -9 & 0 \\
-12 & -6 & -1
\end{pmatrix}
\rangle \\
\Stab(x_4)=& \langle \begin{pmatrix}
125 & 0 & -28 \\
774 & -1 & -172 \\
558 & 0 & -125
\end{pmatrix}
\rangle & \Stab(x_9)=& \langle \begin{pmatrix}
1385 & 528 & -792 \\
1974 & 751 & -1128 \\
3738 & 1424 & -2137
\end{pmatrix}
\rangle
\end{align*}
In addition, we find $16$ connecting elements
{\allowdisplaybreaks
\begin{align*}
c_{1,1}=& \begin{pmatrix}
31 & -5 & 10 \\
96 & -15 & 32 \\
-48 & 8 & -15
\end{pmatrix}
& c_{4,1}=& \begin{pmatrix}
125 & 0 & -28 \\
-24 & 1 & 4 \\
-308 & 0 & 69
\end{pmatrix}\\
c_{4,4}=& \begin{pmatrix}
561 & -28 & -84 \\
3920 & -195 & -588 \\
2440 & -122 & -365
\end{pmatrix}
& c_{4,4}'=& \begin{pmatrix}
145 & 9 & -44 \\
966 & 59 & -292 \\
642 & 40 & -195
\end{pmatrix}
\\
c_{4,8}=&\begin{pmatrix}
1005 & 5 & -232 \\
-1784 & -9 & 412 \\
-1048 & -6 & 243
\end{pmatrix}
& c_{5,5}=& \begin{pmatrix}
289 & 180 & -180 \\
-96 & -59 & 60 \\
368 & 230 & -229
\end{pmatrix}
\\
c_{5,5}'=&\begin{pmatrix}
51 & 35 & -30 \\
-40 & -27 & 24 \\
40 & 28 & -23
\end{pmatrix}
& c_{5,8}=&\begin{pmatrix}
109 & 57 & -68 \\
-186 & -97 & 116 \\
-82 & -42 & 51
\end{pmatrix}
\\
c_{6,6}=&\begin{pmatrix}
139 & 63 & -84 \\
120 & 53 & -72 \\
320 & 144 & -193
\end{pmatrix}
& c_{7,9}=&\begin{pmatrix}
591 & 432 & -224 \\
880 & 643 & -332 \\
1620 & 1184 & -613
\end{pmatrix}
\\
c_{7,7}=&\begin{pmatrix}
45 & 33 & -22 \\
-56 & -41 & 28 \\
8 & 6 & -3
\end{pmatrix}
& c_{8,5}=&\begin{pmatrix}
75 & 51 & -16 \\
26 & 17 & -4 \\
142 & 96 & -29
\end{pmatrix}
\\
c_{8,4}=&\begin{pmatrix}
9 & 5 & 0 \\
70 & 39 & 0 \\
30 & 16 & 1
\end{pmatrix}
& c_{9,7}=&\begin{pmatrix}
231 & 96 & -136 \\
-374 & -155 & 220 \\
-118 & -48 & 69
\end{pmatrix}
\\
c_{9,9}=&\begin{pmatrix}
1575 & 603 & -902 \\
2506 & 961 & -1436 \\
4422 & 1694 & -2533
\end{pmatrix}
& c_{9,6}=& \begin{pmatrix}
1385 & 528 & -792 \\
1974 & 751 & -1128 \\
3738 & 1424 & -2137
\end{pmatrix}.
\end{align*}
} 
The notation shall be understood that via the connecting element $c_{i,j}$ the $D$-perfect points $x_i$ and $x_j$ are adjacent in the residue class graph, in other words $x_i$ and $x_jc_{i,j}^{-1}$ are contiguous. The precise contiguity relations between the points are represented in the residue class graph $\Gamma_D/\Omega$ in \Cref{figResGraph}. The straight lines represent direct contiguity, the curved ones contiguity by the connecting elements which are labeled on the edges.
\begin{center}
\begin{figure}[h!]
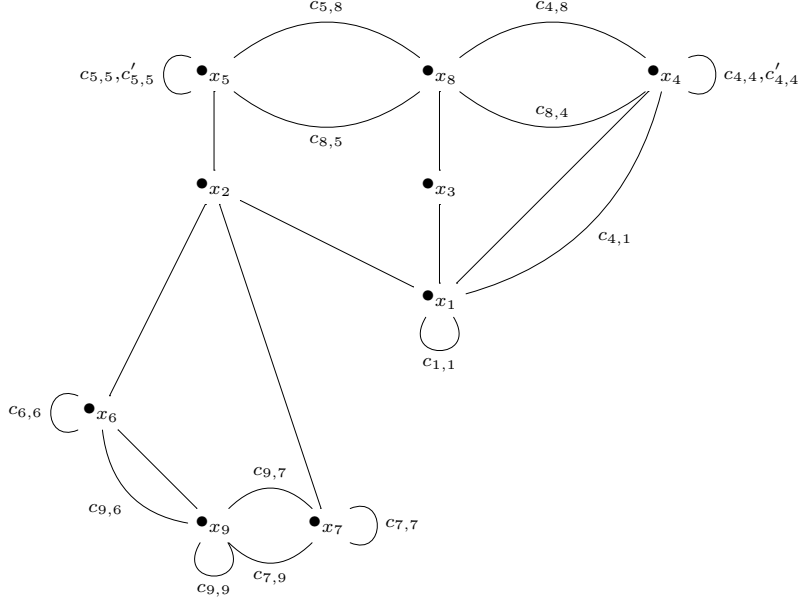

\xygraph{
                         !{<0cm,0cm>;<1.5cm,0cm>:<0cm,1.5cm>::}
                         !{(0,0)}*+{\bullet_{x_5}}="x_5"
                         !{(2,0)}*+{\bullet_{x_8}}="x_8"
                         !{(4,0)}*+{\bullet_{x_4}}="x_4"
                         !{(0,-1)}*+{\bullet_{x_2}}="x_2"
                         !{(2,-1)}*+{\bullet_{x_3}}="x_3"
                         !{(2,-2)}*+{\bullet_{x_1}}="x_1"
                         !{(-1,-3)}*+{\bullet_{x_6}}="x_6"
                         !{(1,-4)}*+{\bullet_{x_7}}="x_7"
                         !{(0,-4)}*+{\bullet_{x_9}}="x_9"
                         "x_5"-"x_2"
                         "x_8"-"x_3"
                         "x_4"-"x_1"
                         "x_2"-"x_1"
                         "x_2"-"x_6"
                         "x_2"-"x_7"
                         "x_3"-"x_1"
                         "x_6"-"x_9"
                         "x_5" -@(ul,dl) "x_5" _{c_{5,5},c_{5,5}'}
                         "x_4" -@(ur,dr) "x_4" ^{c_{4,4},c_{4,4}'}
                         "x_5" -@/^0.7cm/ "x_8" ^{c_{5,8}}
                         "x_5" -@/_0.7cm/ "x_8" _{c_{8,5}}
                         "x_8" -@/^0.7cm/ "x_4" ^{c_{4,8}}
                         "x_8" -@/_0.7cm/ "x_4" ^{c_{8,4}}
                         "x_4" -@/^0.7cm/ "x_1" ^{c_{4,1}}
                         "x_1" -@(dl,dr) "x_1" _{c_{1,1}}
                         "x_7" -@/^0.5cm/ "x_9" ^{c_{7,9}}
                         "x_7" -@/_0.5cm/ "x_9" _{c_{9,7}}
                         "x_6" -@(ul,dl) "x_6" _{c_{6,6}}
                         "x_9" -@(dl,dr) "x_9" _{c_{9,9}}
                         "x_7" -@(ur,dr) "x_7" ^{c_{7,7}}
                         "x_6" -@/_0.5cm/ "x_9" _{c_{9,6}}
                         }                        
                         \label{figResGraph}
                         \caption{Residue class graph $\Gamma_D/\Omega$} 
\end{figure}
\end{center}
Calculating this information takes about $6$ seconds.
\end{example}
In \Cref{exComplete} it does not matter whether one applies the Watson process since it does not have any effect on the matrix and runs in virtually no time. The following example shall illustrate the situation when the Watson process does change things.
\begin{example}\label{exWatson}\mbox{}\\
In this example we want to demonstrate how effective the usage of the Watson process can be. For that purpose, we look at the matrix
\[A=\begin{pmatrix}
17 & -17 & 20 & -9 \\
-17 & -25 & 15 & -6 \\
20 & 15 & 4 & -2 \\
-9 & -6 & -2 & 1
\end{pmatrix}
,\quad 
A^{-1}=\frac{1}{32}\begin{pmatrix}
9 & -6 & 91 & 227 \\
-6 & 4 & -50 & -130 \\
91 & -50 & 1137 & 2793 \\
227 & -130 & 2793 & 6881
\end{pmatrix}
.
\]
We have $\det A=-2^5$ and the genus symbol (in the notation of \cite{CS}, neglecting the additional parameters of the $2$-adic genus symbol) of $A$ is given by
\[(1^3\cdot 32),\]
thus the Watson process does have an effect and yields
\[\Watson(A)=\begin{pmatrix}
-8 & -1 & -2 & -19 \\
-1 & 10 & -15 & -4 \\
-2 & -15 & 21 & -2 \\
-19 & -4 & -2 & -25
\end{pmatrix}
\]
with $\det \Watson(A)=-2^3$.

The direct calculation of the automorphism group of $A$ takes more than $1$ hour while using the Watson process gives the result in about $14$ minutes, of which the orbit-stabilizer calculation took next to no time. For $A$, the algorithm finds $19$ inequivalent $D$-perfect points, for $\Watson(A)$ there are only $3$. This and the reduction of the determinant by a factor $2^2$ explains this difference.
\end{example}
After these rather random examples we also give some results for simpler Gram matrices. 
\begin{example}
The standard hyperbolic form
\[H^{(1)}_n:=diag(-1,1,\dots ,1)\in\Zhyp\]
can be handled by our algorithm in the current implementation for $n\leq 8$. For $n\leq 4$, there is exactly one $D$-perfect point 
\[x_1=(1,0,\dots ,0)\]
with $2^{n-1}$ directions for $n\geq 3$ and no non-blind direction for $n=2$. 

For $n\geq 5$, we have the same point with the same number of directions and one additional point
\[x_2\in\{(3,-1,1,-1,1),(3,-1,1,1,-1,-1),(3,1,1,-1,-1,1,1),(3,-1,1,1,1,1,-1,1) \}\]
with $5,32,99, and 632$ directions respectively. 

Computing this takes much less than one second for $n\leq 5$ and for $n=6$ ($n=7,n=8$) it takes about $2$ seconds ($1$ minute, $3$ minutes). 

For $n\geq 9$ we eventually run out of memory.
\end{example}
We now consider hyperbolic lattices arising from graphs. The entries of the matrix corresponding to a graph  are defined as $2$ on the diagonal and $-1$ on position $(i,j)$ if vertices $i$ and $j$ are connected by an edge.
\begin{example}
\begin{enumerate}[(i)]
\item The matrices of complete graphs (in the above sense) define hyperbolic matrices if the graph has at least $4$ vertices (cf. \Cref{figComplGraphs}). 
\begin{figure}[h!]
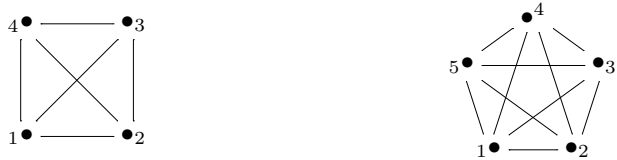

\begin{minipage}[c]{0.5\textwidth}
\[\xygraph{
                         !{<0cm,0cm>;<1.5cm,0cm>:<0cm,1.5cm>::}
                         !{(0,0)}*+{{}_{1}\bullet}="1"
                         !{(1,0)}*+{\bullet_{2}}="2"
                         !{(1,1)}*+{\bullet_{3}}="3"
                         !{(0,1)}*+{{}_{4}\bullet}="4"
                         "1"-"2"
                         "1"-"3"
                         "1"-"4"
                         "2"-"3"
                         "2"-"4"
                         "3"-"4"
                         }\]
                         
                         \end{minipage}
                         \begin{minipage}[c]{0.5\textwidth}

\[\xygraph{
                         !{<0cm,0cm>;<1cm,0cm>:<0cm,1cm>::}
                         !{(-0.59,-0.81)}*+{{}_{1}\bullet}="1"
                         !{(0.59,-0.81)}*+{\bullet_{2}}="2"
                         !{(0.95,0.31)}*+{\bullet_{3}}="3"
                         !{(0,1)}*+{\bullet^{4}}="4"
                         !{(-0.95,0.31)}*+{{}_{5}\bullet}="5"
                         "1"-"2"
                         "1"-"3"
                         "1"-"4"
                         "1"-"5"
                         "2"-"3"
                         "2"-"4"
                         "2"-"5"
                         "3"-"4"
                         "3"-"5"
                         "4"-"5"
                         }\]
                         
                         \end{minipage}
                         \label{figComplGraphs}
                         \caption{Complete graphs with $4$ and $5$ vertices}
\end{figure}
For the corresponding hyperbolic lattice our algorithm yields in dimension $4$ ($5$) $1$ $D$-perfect point $x_1=(1,\dots ,1)$ within much less than $1$ (about $17$) seconds. For the complete graph of order $6$ the algorithm needs too much memory.
\item Consider the hyperbolic lattice to the graph
\[\xygraph{
                         !{<0cm,0cm>;<1cm,0cm>:<0cm,1cm>::}
                         !{(0,-1)}*+{\bullet_1}="1"
                         !{(0.86,-0.5)}*+{\bullet_2}="2"
                         !{(0.86,0.5)}*+{\bullet_3}="3"
                         !{(0,1)}*+{\bullet_4}="4"
                         !{(-0.86,0.5)}*+{\bullet_5}="5"
                         !{(-0.86,-0.5)}*+{\bullet_6}="6"
                         "1"-"2"
                         "1"-"3"
                         "1"-"6"
                         "1"-"5"
                         "2"-"3"
                         "2"-"4"
                         "2"-"6"
                         "3"-"4"
                         "3"-"5"
                         "4"-"5"
                         "4"-"6"
                         "5"-"6"
                         }\]
For this one we obtain $2$ $D$-perfect points
\[x_1=(0,0,1,1,1,1),\quad x_2=(-1,1,4,3,3,2)\]
within less than $5$ minutes.

The residue class graph looks like this:
\[\xygraph{
                         !{<0cm,0cm>;<2cm,0cm>:<0cm,2cm>::}
                         !{(0,0)}*+{\bullet_{x_1}}="x1"
                         !{(2,0)}*+{\bullet_{x_2}}="x2"
                         "x1"-"x2"
                         "x1"-@(ul,dl) "x1"_{c_{1,1}}
                         "x1"-@/^0.7cm/ "x2"^{c_{1,2}}
                         "x2"-@(ur,dr) "x2"^{c_{2,2},c_{2,2}'}
                         }.\]
\end{enumerate}
\end{example}
\subsection{Statistics}\label{subStat}
Since we cannot give a rigorous running time analysis of our algorithm, we give some statistics to illustrate the algorithm's performance. For that we randomly chose $1000$ hyperbolic $3\times 3$ matrices and measured the time needed to perform our algorithm completely (see \Cref{subWatson}).
\begin{figure}[h!]
\begin{minipage}[t]{0.5\textwidth}
\includegraphics[scale=0.5]{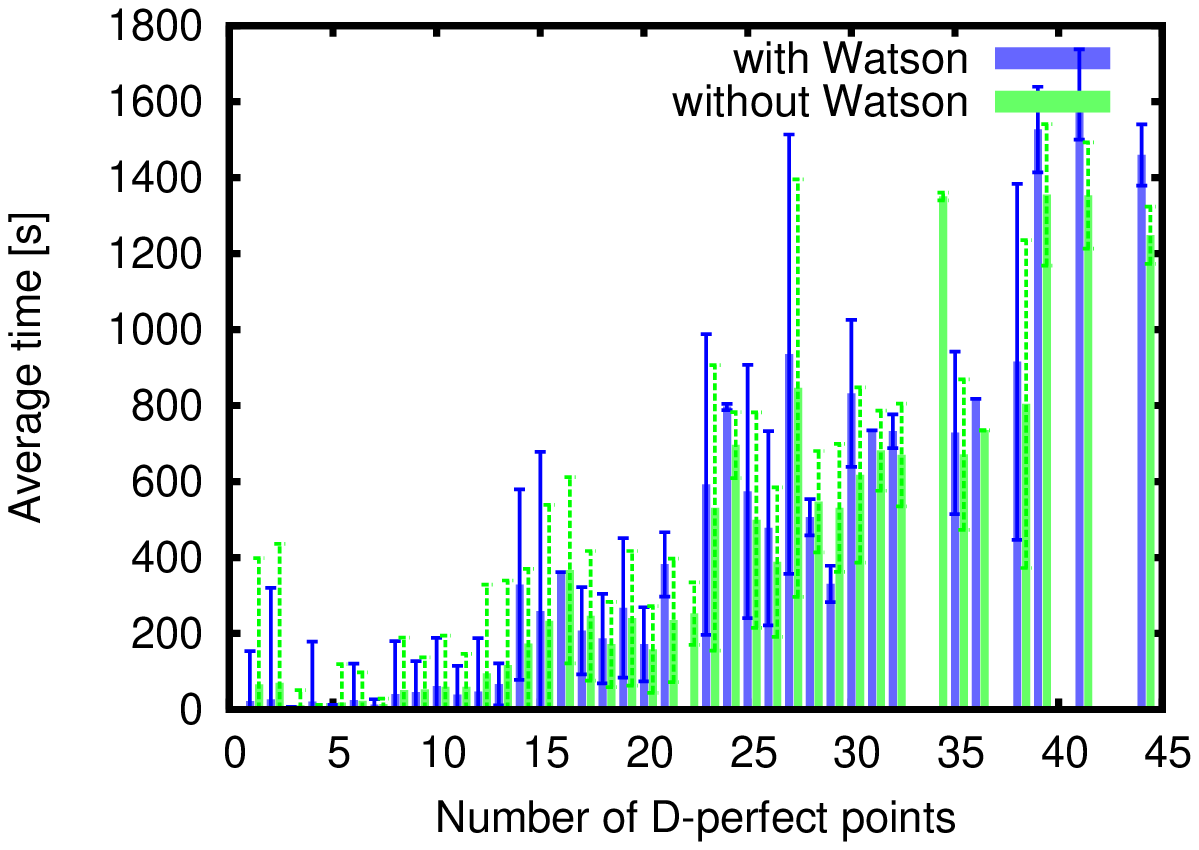}
\end{minipage}
\begin{minipage}[t]{0.5\textwidth}
\includegraphics[scale=0.5]{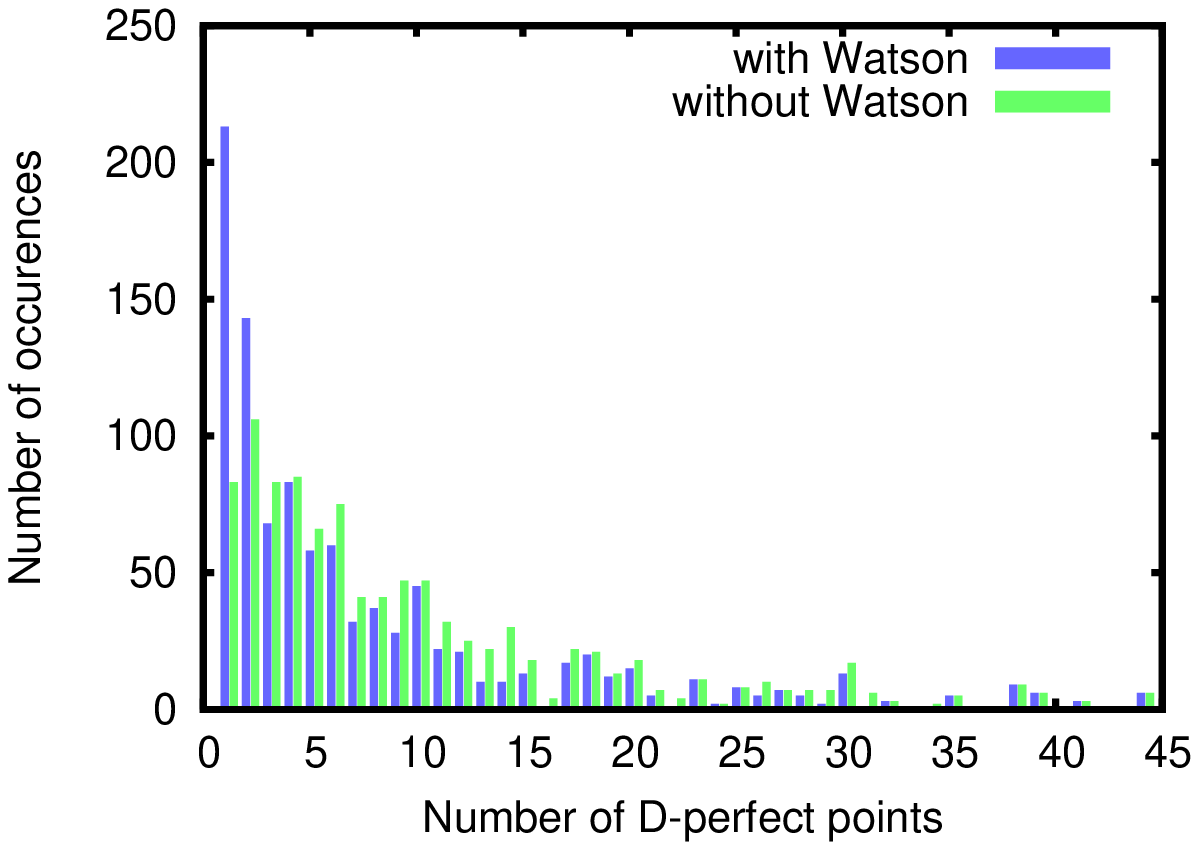}
\end{minipage}
\caption{Time and point number distribution with and without the Watson process}
\label{figWatson}
\end{figure} 
As we can see, the average time in dependance of the number of $D$-perfect points is comparably increasing in both cases, while in case of the usage of the Watson process we get many more examples with very few perfect points as predicted in \Cref{subWaston}.

We also did this with the additional condition that the quadratic form induced by $A$ should be isotropic resp. anisotropic. \Cref{figIso,figAniso} below show the statistical results. 
\begin{figure}[h!]
\begin{minipage}[t]{0.5\textwidth}
\includegraphics[scale=0.5]{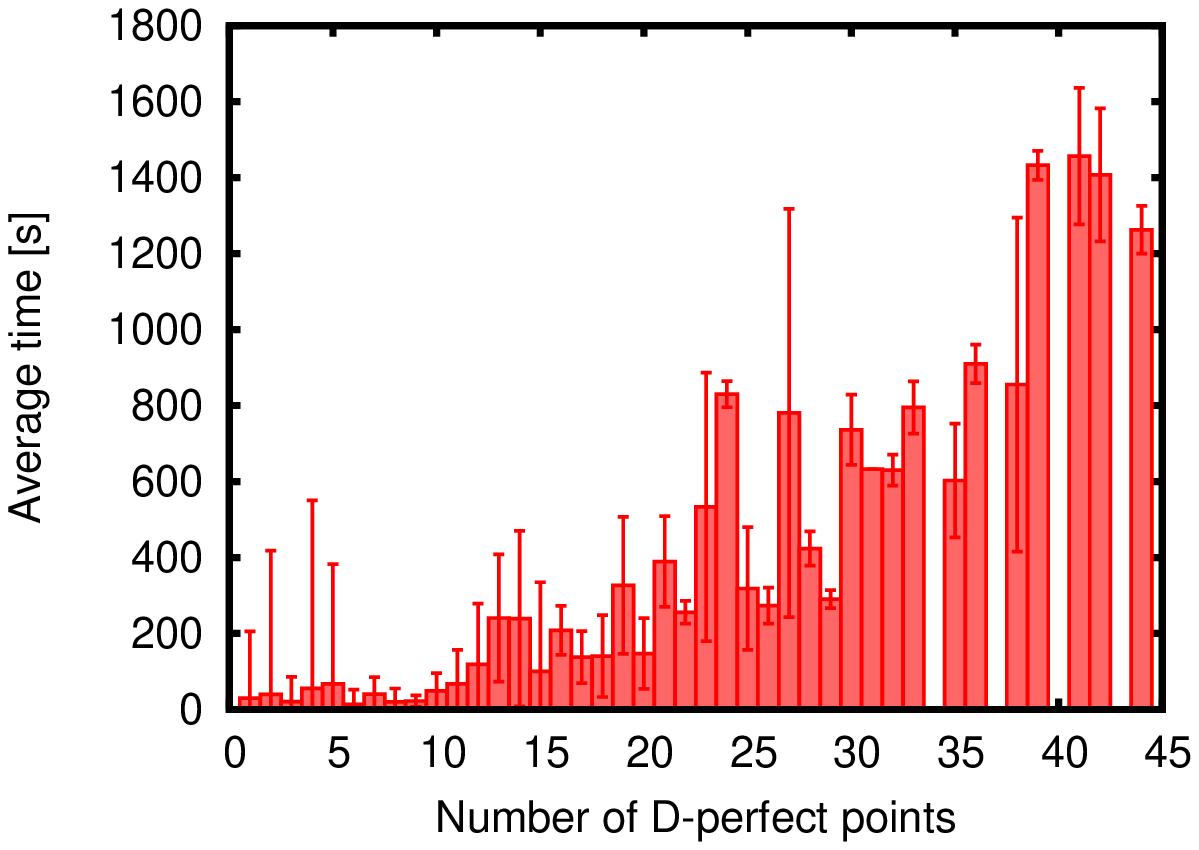}
\end{minipage}
\begin{minipage}[t]{0.5\textwidth}
\includegraphics[scale=0.5]{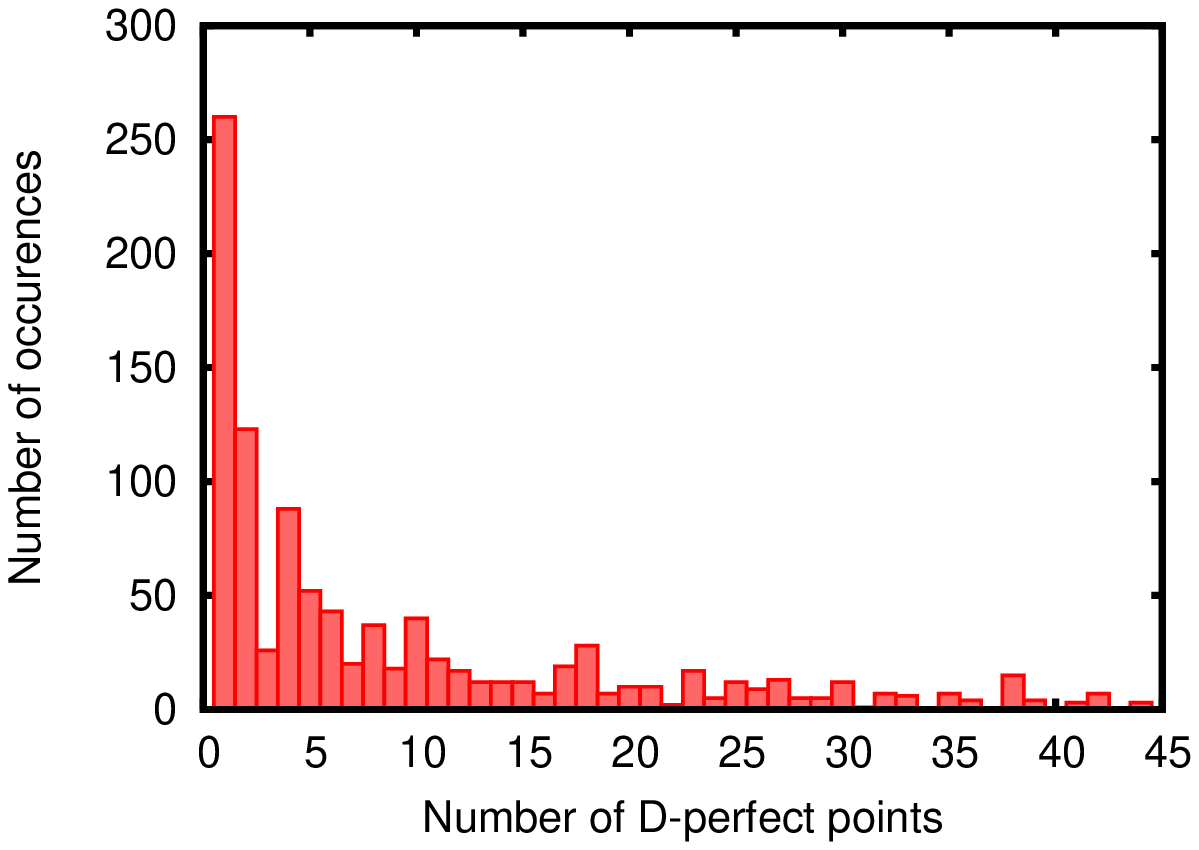}
\end{minipage}
\caption{Time and point number distribution for isotropic forms}
\label{figIso}
\end{figure} 
 
\begin{figure}[h!]
\begin{minipage}[t]{0.5\textwidth}
\includegraphics[scale=0.5]{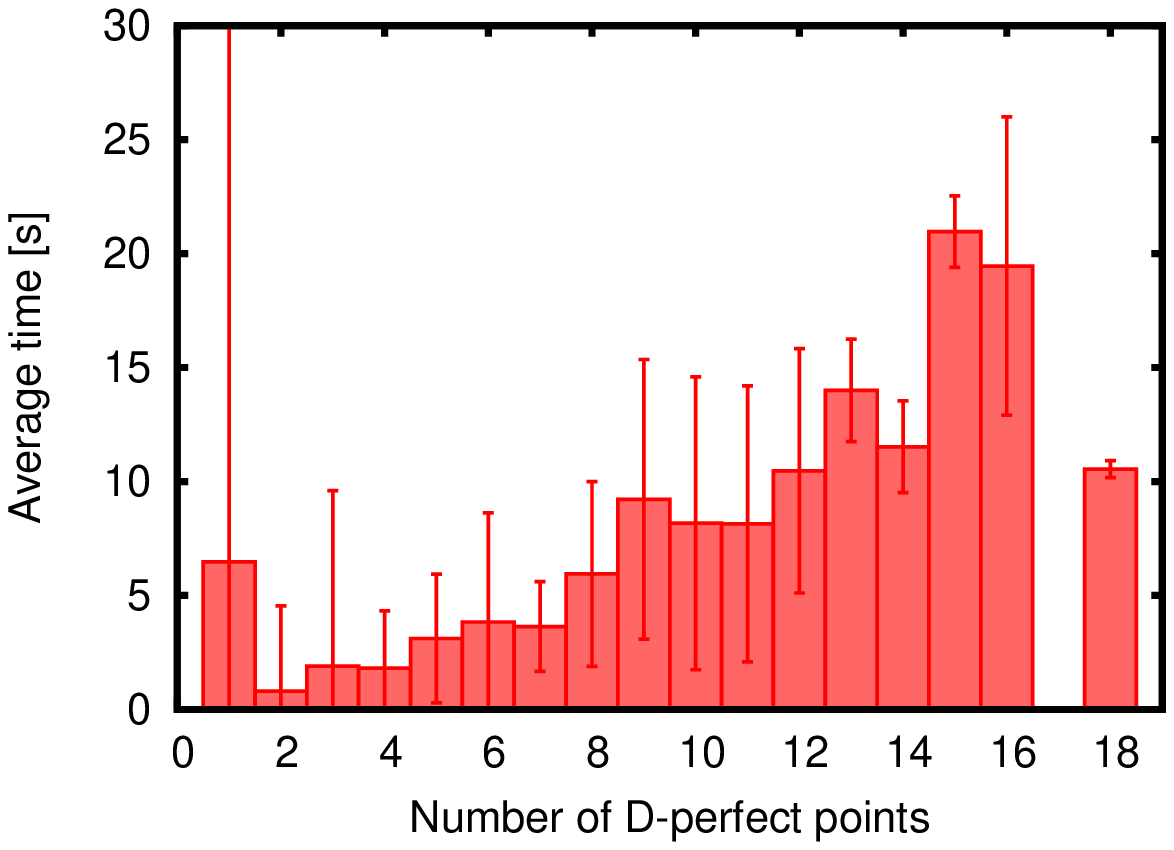}
\end{minipage}
\begin{minipage}[t]{0.5\textwidth}
\includegraphics[scale=0.5]{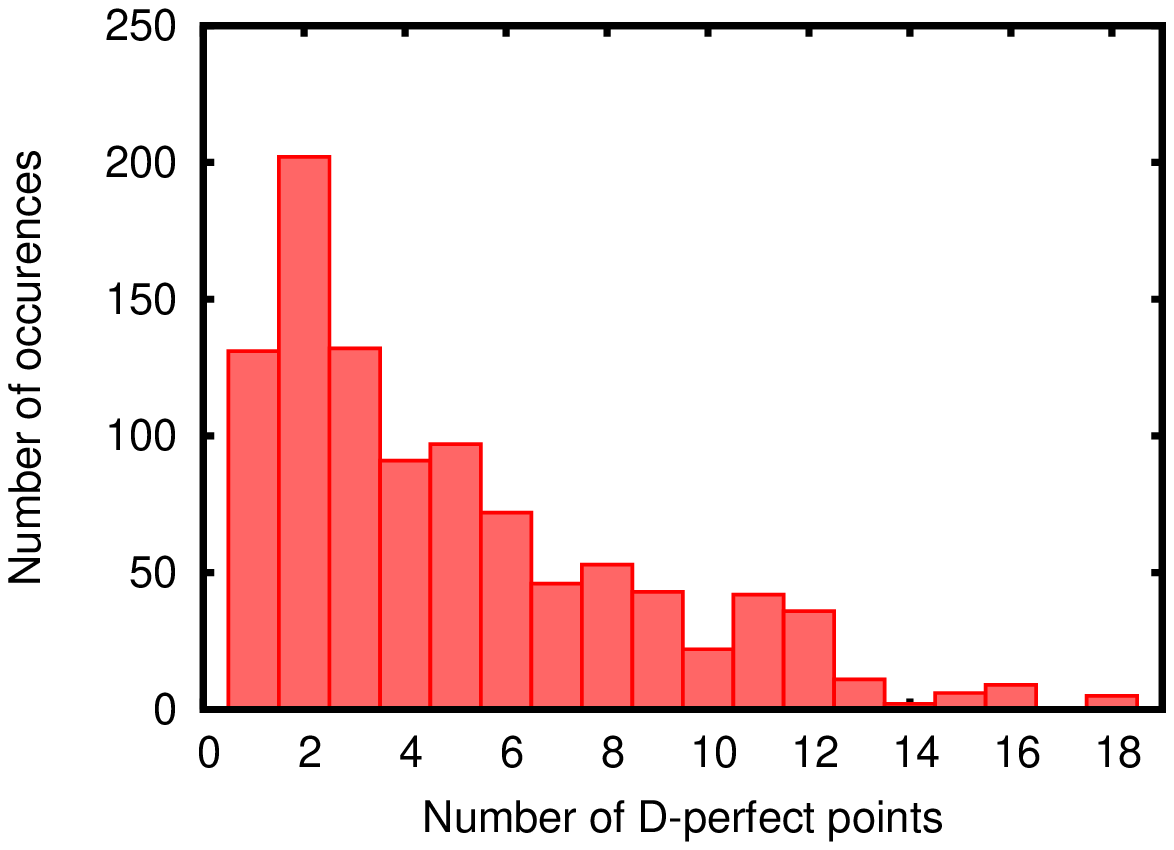}
\end{minipage}
\caption{Time and point number distribution for anisotropic forms}
\label{figAniso}
\end{figure} 
 
Apparently, automorphism groups of anisotropic lattices are - in a way - far easier to compute than the ones of isotropic lattices. This phenomenon occured  while testing the implementation of our algorithm, but we cannot really explain it. However, this also shows the significance of the usage of the Watson process, because, heuristically, the Watson process affects about $50\%$ of all isotropic forms of rank $3$, while it only affects about $30\%$ of all anisotropic forms and the computation for anisotropic forms is much faster anyway. On the other hand this also gives an indication towards why the average running time of our algorithm gets so much higher in dimensions $\geq 5$, since by Meyer's Theorem in these dimensions every hyperbolic form is isotropic.
\section*{Acknowledgements}
Some of the ideas presented in \Cref{secHyp} of this paper are already due to Opgenorth and are contained in preliminary and unpublished versions of \cite{Opg01} which the author was kindly allowed to use. 

He wants to thank Prof. Gabriele Nebe and Markus Kirschmer for reading a first draft of this paper and their helpful suggestions, David Lorch and again Markus Kirschmer for their help with the \textsc{Magma} implementation of the algorithm, and the anonymous referee whose comments helped to improve several parts of this paper.

\end{document}